\documentclass{amsart}
\usepackage[utf8]{inputenc}
\usepackage{tikz}
\usepackage{amsthm}
\usepackage[shortlabels]{enumitem}
\usepackage{leftindex}
\usepackage{amsmath}
\usepackage{amssymb}
\usepackage{amsthm}
\usepackage{amscd}
\usepackage{mathtools}
\usepackage{stmaryrd}
\usepackage{comment}
\usepackage{tikz-cd}
\usepackage{booktabs}
\usepackage{mathdots}
\usepackage{array}
\usepackage{blkarray}
\usepackage{multirow}
\newcolumntype{L}[1]{>{\raggedright\arraybackslash}p{#1}}

\usepackage{biblatex}
\usepackage{hyperref}
\newtheorem{thm}{Theorem}[section]
\newtheorem*{thm*}{Theorem}
\newtheorem{lem}[thm]{Lemma}

\newtheorem{prop}[thm]{Proposition}

\theoremstyle{definition}
\newtheorem{definition}[thm]{Definition}

\newtheorem{remark}[thm]{Remark}

\newtheorem{example}[thm]{Example}

\numberwithin{equation}{section}

\newcommand{\kk}{\Bbbk}

\newcommand\xto[1]{\xrightarrow{#1}}
\newcommand{\mf}{\mathfrak}
\newcommand{\mb}{\mathbb}
\newcommand{\mc}{\mathcal}
\newcommand{\mbf}{\mathbf}

\DeclareMathOperator\rank{rank}

\DeclareMathOperator\SL{SL}
\DeclareMathOperator\GL{GL}
\DeclareMathOperator\SO{SO}
\DeclareMathOperator\OG{OG}
\DeclareMathOperator\Spin{Spin}

\DeclareMathOperator\Spec{Spec}
\DeclareMathOperator\depth{depth}

\DeclareMathOperator\grade{grade}
\DeclareMathOperator\ann{ann}
\DeclareMathOperator\pdim{pdim}
\DeclareMathOperator\Span{span}

\DeclareMathOperator\Ext{Ext}

\DeclareMathOperator\Hom{Hom}

\newcommand{\ssc}{\mathrm{ssc}}

\newcommand{\todo}[1]{{\color{blue}{(#1)}}}

\author{Lorenzo Guerrieri}
\address{}
\curraddr{}
\email{}
\thanks{}

\author{Tymoteusz Chmiel}
\address{}
\curraddr{}
\email{}
\thanks{}

\author{Xianglong Ni}
\address{}
\curraddr{}
\email{}
\thanks{}

\author{Jerzy Weyman}
\address{}
\curraddr{}
\email{}
\thanks{}

\begin{document}
	
	\title{Grade three perfect ideals and length four self-dual resolutions}
	
	\maketitle
	
	\begin{abstract}
		Starting with a grade three perfect ideal $I \subset R$, we demonstrate how to produce the a self-dual resolution of length four using the resolution of the original ideal. This process is also reversible. The main case of interest is when the grade three perfect ideal has type two, so the output complex resolves $R/J$ for a grade four Gorenstein ideal $J$. This suggests that the structure theory of these two families of ideals should be closely related.
	\end{abstract} 
	
	%\setcounter{tocdepth}{2}
	%\tableofcontents
	
	\section{Introduction}\label{sec:intro}
	Although understanding the structure theory of ideals in complete generality seems hopelessly out of reach, the situation for perfect ideals appears more tractable, and has been marked with a number of important successes in the previous century. Most well-known are the cases of grade two perfect ideals and grade three Gorenstein ideals. The former comes as an easy consequence of the Hilbert-Burch theorem (\cite{Hilbert1890} and \cite{Burch68}) on resolutions of cyclic modules with projective dimension two. Explicitly, every grade two perfect ideal is generated by the maximal minors of a $(n-1)\times n$ matrix. The structure theory of grade three Gorenstein ideals was determined by Buchsbaum and Eisenbud in \cite{Buchsbaum-Eisenbud77}, in which they show that any such ideal is generated by the submaximal pfaffians of a $n \times n$ skew matrix, where $n$ is odd.
	
	Importantly, in each of the two aforementioned cases, taking a generic $(n-1)\times n$ matrix (resp. $n \times n$ skew matrix) yields a grade two perfect (resp. grade three Gorenstein) ideal. Thus, the structure theorems identify the generic examples for each family, and consequently they are ``optimal'' in the sense that there cannot be any further necessary relations among the entries of the (skew) matrix.
	
	When we discuss structure theorems in the following, we refer to theorems of this type---ones that identify a generic example for a family, from which all other examples arise as specializations. With the case of grade two completely settled, as well as the case of grade three Gorenstein, the two next-simplest cases are: (1) grade three perfect ideals of type at least two, and (2) grade four Gorenstein ideals.
	
	Progress in the former direction has primarily been aided by linkage. For example, the structure theory of grade three almost complete intersections was deduced as a consequence of the theorem for Gorenstein ideals in \cite{Buchsbaum-Eisenbud77} in this manner. Continuing this, structure theorems for certain grade three perfect ideals directly linked to almost complete intersections were given in \cite{Brown87} and \cite{Sanchez89}. These examples are not exhaustive: even for deviation and type two, there exist ideals not covered by the structure theorem in \cite{Brown87}, see for instance \cite{CLKW20}. Recently, a representation-theoretic approach was used in \cite{GNW-ADE} to give structure theorems for all grade three perfect ideals of small type and deviation, but the general story remains largely mysterious.
	
	For grade four Gorenstein ideals, \cite{Herzog-Miller85} and \cite{Vasconcelos-Villarreal86} showed under mild hypotheses that any such ideal with at most six generators is a hypersurface section. However, this is not true if one increases the number of generators; examples generated by any odd $n \geq 7$ number of elements are given in \cite{Kustin-Miller82}, and there are also the well-known determinantal examples of \cite{Gulliksen-Negard72}. An analysis of the resolutions of grade four Gorenstein ideals is given in \cite{Reid15}, but unfortunately it does not result in a theorem in the style of Buchsbaum-Eisenbud.
	
	At a glance, it is not easy to say which of these two directions ought to exhibit a simpler structure theory. From the perspective of free resolutions, compared to the resolution of a grade three perfect ideal, there are fewer matrices to prescribe for the resolution of a grade four Gorenstein ideal owing to its self-duality (i.e. two matrices instead of three).

	On the other hand, there are various constructions which take a grade $c$ perfect ideal and produce a grade $c+1$ Gorenstein ideal. To name a few:
	\begin{itemize}
		\item One can take the sum of two geometrically linked ideals, as demonstrated in \cite{Peskine-Szpiro74} and studied further in \cite{Ulrich90}.
		\item There is the ``doubling'' construction, based on the existence of a canonical ideal, given in \cite[Lemma 3.2]{Kustin-Miller82}.
		\item There is the ``big from small'' construction of \cite{Kustin-Miller83}, later dubbed unprojection and analyzed further in \cite{Papadakis-Reid04}.
	\end{itemize}
	The existence of these constructions suggests that the structure theory of grade four Gorenstein ideals is at least as complicated as the situation for grade three perfect ideals.
	
	Our contribution in this paper is another construction, distinct from the preceding, which (in a special case) produces a grade four Gorenstein ideal from a grade three perfect ideal of type two. Despite its narrow scope, this construction has a major advantage: it is \emph{reversible}. Consequently, this construction could potentially be used to translate structure theorems from one family of ideals to the other, an idea which has traditionally been enabled by linkage. Informally, it also suggests that these two apparently orthogonal directions have comparable difficulty---unfortunately this difficulty appears to be quite high either way!
	
	The organization of this paper is as follows.%In \S\ref{sec:motivation}, we give the representation-theoretic motivation for our main construction and draw parallels to the familiar theory of linkage.
	We fix notation and give the necessary commutative algebra background in \S\ref{sec:background}. After that, we turn to our main construction: in \S\ref{sec:grade three to four} we demonstrate how to produce a self-dual resolution of length four starting from a grade three perfect ideal, and in \S\ref{sec:grade four to three} we show how to reverse this procedure. Finally, we conclude with some examples in \S\ref{sec:examples}, where we also discuss the primary application of this machinery to \cite{GNW-ADE} in a future work: the classification of all grade four Gorenstein ideals on up to eight generators.
	
	\subsection*{Acknowledgments}
	Work on this paper began while the authors were in residence at the Simons Laufer Mathematical Sciences Institute (formerly MSRI) in Berkeley, California, during the Spring 2024 semester. The authors would like to thank Ela Celikbas, Lars Christensen, David Eisenbud, Sara Angela Filippini, Craig Huneke, Witold Kraskiewicz, Andrew Kustin, Jai Laxmi, Claudia Polini, Steven Sam, Frank-Olaf Schreyer, Jacinta Torres, Bernd Ulrich, and Oana Veliche for interesting discussions pertaining to this paper and related topics.

	\section{Notation and background}\label{sec:background}
	\subsection{The ambient ring}
	Throughout, all rings considered are commutative and Noetherian. The primary case of interest is when they are furthermore either local or graded. In the latter case, we consider homogeneous ideals and their graded free resolutions.
	
	All modules considered are finitely generated. For an $R$-module $M$, its grade is defined to be
	\[
		\grade M \coloneqq \inf_i \{ \Ext^i(M,R) \neq 0\}.
	\]
	If $I \subset R$ is an ideal, it is a customary abuse of notation to use $\grade I$ to denote $\grade R/I$. By the above definition, the grade of the unit ideal is $\infty$. We say the module $M$ is \emph{perfect} if $\pdim_R M = \grade M$. This is equivalent to stating that $M$ admits a finite free resolution whose dual is also acyclic; the dual then resolves $\Ext^c(M,R)$ where $c = \grade M$. Again, we say an ideal $I$ is perfect if the module $R/I$ is.
	
	We say that a grade $c$ perfect module $M$ is self-dual if $\Ext^c(M,R) \cong M$. If furthermore $M \cong R/J$ for an ideal $J$, we say that the ideal $J$ is \emph{Gorenstein}. Note that if the ambient ring $R$ were itself a Gorenstein ring, then this would exactly be the case in which the quotient $R/J$ is also.
	
	\subsection{Notation for free resolutions}\label{subsec:A-B-notation}
	The \emph{format} of a finite free resolution $\mb{F}$ is the sequence of ranks $(\rank F_0, \rank F_1,\ldots)$.
	
	We fix two integers $p \geq 1$ and $q \geq 3$. We will be concerned with resolutions of length three with format $(1,p+2, p+q,q-1)$, and self-dual resolutions of length four with format $(q-2,p+q,2(p+2),p+q,q-2)$. For the former, we will use the following notation for the complex:
	\[
		\mb{A} \colon 0 \to A_3 \xto{d_3} A_2 \xto{d_2} A_1 \xto{d_1} A_0.
	\]
	Here $A_0 = R$, $A_1= R^{p+2}$, $A_2= R^{p+q}$, and $A_3 = R^{q-1}$.
	
	We will also deal with self-dual resolutions of length four of the form
	\[
		\mb{B} \colon 0 \to B_0^* \xto{\delta_4 = \delta_1^*} B_1^* \xto{\delta_3 = \delta_2^*} B_2^* \cong B_2 \xto{\delta_2} B_1 \xto{\delta_1} B_0.
	\]
	Here $B_0 = R^{q-2}$, $B_1 = R^{p+q}$, and $B_2 = H \oplus H^*$ where $H=R^{p+2}$ and the isomorphism $B_2^* \cong B_2$ is the evident one.
	
	If $\mb{B}$ resolves $M$, then $M$ is a grade 4 self-dual module. However, it is not clear whether every such module admits an explicitly self-dual resolution of the given form. This is known under the assumptions $M$ is cyclic, $1/2\in R$, and either:
	\begin{itemize}
		\item \cite{Celikbas-Laxmi-Weyman-23} $R = \kk[x_1,\ldots,x_n]$ is a graded polynomial ring over a field $\kk$, with all variables of positive degree, and $M$ is a graded module, or
		\item \cite{Kustin-Miller80} $R$ is a complete regular local ring.
	\end{itemize}
	\subsection{Verifying acyclicity}\label{sec:bg-ffrBE}
	We recall the Buchsbaum-Eisenbud acyclicity criterion:
	\begin{thm}[\cite{Buchsbaum-Eisenbud73}]
		A complex of free $R$-modules
		\[
			0 \to F_c \xto{d_c} F_{c-1} \xto{d_{c-1}} \cdots \xto{d_2} F_1 \xto{d_1} F_0
		\]
		is exact if and only if, for all $i$ from 1 to $c$,
		\begin{itemize}
			\item $\rank d_i + \rank d_{i+1} = \rank F_i$, and
			\item $\grade I_{\rank d_i}(d_i) \geq i$.
		\end{itemize}
	\end{thm}
	Let $s \geq 0$ be an integer. We say something holds \emph{in grade $s$} if it is true over all localizations $R_\mf{p}$ where $\mf{p}$ is a prime with $\grade \mf{p} \leq s$. For example, a relation such as $d^2 = 0$ holds if and only if it does in grade zero, i.e. generically.
	
	A complex $\mb{F}$ is acyclic in grade $s$ exactly when the first condition above holds and we have $\grade I_{\rank d_i}(d_i) \geq \min(i,s+1)$. The acyclicity of finite free complexes of length $c$ can therefore be checked in grade $c-1$.
	
	\iffalse
	\subsection{The standard split complex}
	Throughout the paper, we will often need to establish certain relations among maps constructed from finite free resolutions. In some cases (e.g. Lemma~\ref{lem:B complex}) this may be easy to do directly, but for more complicated relations we will resort to the following strategy.
	
	If $\mb{F}$ is a finite free resolution, then in particular it resolves a module $M$ of positive grade. Taking $h \in \ann M$ to be any non-zerodivisor, the localization $R \to R_h$ is an inclusion, so it suffices to verify the desired relations for the split exact complex $\mb{F}\otimes R_h$.
	
	For the length three case, if the relations to be verified are preserved by the evident actions of $\GL(A_i)$ on the resolution, then we can reduce to checking them for the following single example:
	\begin{definition}
		$\mb{A}^\ssc$ is \todo{finish this...}
	\end{definition}
	
	For the length four case, if the relations to be verified are preserved by the evident actions of $\GL(B_1)$ and $\SO(B_2)$ on the resolution, then we can reduce to checking them for the following single example:
	\begin{definition}
		$\mb{B}^\ssc$ is
		\todo{finish this...}
	\end{definition}
	\fi
	
	\section{From $\mb{A}$ to $\mb{B}$}\label{sec:grade three to four}
	Let
	\[
	\mb{A} \colon 0 \to A_3 \xto{d_3} A_2 \xto{d_2} A_1 \xto{d_1} A_0 \cong R.
	\]
	be a free resolution over the ring $R$. We assume that its dual $\mb{A}^*$ is also acyclic. If $R/I$ is the module resolved by $\mb{A}$, this is equivalent to stating that the ideal $I$ is either a grade three perfect ideal or the unit ideal; of course the former will be the case of interest. The alternating sum $\sum (-1)^i \rank A_i$ is zero, so the format of $\mb{A}$ is $(1,p+2,p+q,q-1)$ for suitable integers $p$ and $q$. For the purposes of the construction, we assume that $p \geq 1$ (which is automatic if $\mb{A}$ is not split) and also that $q \geq 3$.
	
	\subsection{Multiplicative structure on $\mb{A}$}
	We fix a differential graded algebra structure on $\mb{A}$, c.f. \cite{Buchsbaum-Eisenbud77}. This gives us a multiplication $A_m \otimes A_n \to A_{m+n}$ that is graded-commutative and satisfies the Leibniz rule. Since $\mb{A}$ has length three, this multiplication is also strictly associative, but that will not be important for our purposes.
	
	To be explicit, this multiplication can be chosen as follows. One begins by defining the product of degree one elements by taking any lift in the diagram below:
	\[
	\begin{tikzcd}
		0 \ar[r] & A_3 \ar[r] & A_2 \ar[r] & A_1 \ar[r] & R\\
		&& \bigwedge^2 A_1 \ar[u,dashed] \ar[ur]
	\end{tikzcd}
	\] 
	Here the diagonal map sends $e_1 \wedge e_2 \mapsto d_1(e_1) e_2 - d_1(e_2) e_1$. Then one defines the multiplication $A_1 \otimes A_2 \to A_3$ by lifting in the below diagram:
	\[
	\begin{tikzcd}
		0 \ar[r] & A_3 \ar[r] & A_2 \ar[r] & A_1 \ar[r] & R\\
		& A_1 \otimes A_2 \ar[u,dashed] \ar[ur]
	\end{tikzcd}
	\]
	Here the diagonal map sends $e \otimes f \mapsto d_1(e) f + d_2(f) e$ where the term $d_2(f) e$ is computed using the multiplication $\bigwedge^2 A_1 \to A_2$ already defined.
	
	We observe that any two choices of $\bigwedge^2 A_1 \to A_2$ differ by $d_3 b$ where $b$ is some map $\bigwedge^2 A_1 \to A_3$.
	
	\begin{example}\label{ex:split A}
		Let $\mb{A}$ be the split complex
		\begin{equation}\label{eq:split A}
			0 \to A_3 \to A_3 \oplus M \to R \oplus M \to R.
		\end{equation}
		where $M = R^{p+1}$. If $a\in A_3$, $m\in M$, $r \in R$, we let $a_{(i)}$ ($i = 2,3$), $m_{(i)}$ ($i = 1,2$), $r_{(i)}$ ($i=0,1$) denote $a$, $m$, $r$ viewed elements in homological degree $i$. 
		
		A particularly simple choice of multiplication on $\mb{A}$, indeed the unique $\GL(A_3)\times\GL(M)$-equivariant one, is given as follows:
		\begin{itemize}
			\item The product $r_{(1)} \cdot m_{(1)}$ is $m_{(2)}$. All products $m_{(1)} \cdot m'_{(1)}$ are equal to zero.
			\item The product $r_{(1)} \cdot a_{(2)}$ is $a_{(3)}$. All products $r_{(1)} \cdot m_{(2)}$, $m_{(1)} \cdot m'_{(2)}$, $m_{(1)} \cdot a_{(2)}$ are equal to zero.
		\end{itemize}
		
		Let $b$ be a map $\bigwedge^2 A_1 \to A_3$. Then we get a different choice of multiplication on $\mb{A}$ by:
		\begin{itemize}
			\item adding $d_3b(e_1\wedge e_2)$ to all products $e_1 \cdot e_2$ where $e_1,e_2 \in A_1$, and
			\item adding $b(d_2(f) \wedge e)$ to all products $e \cdot f$ where $e\in A_1$ and $f \in A_2$.
		\end{itemize}
		Furthermore, every multiplication on $\mb{A}$ is obtained in this manner for some choice of $b$.
	\end{example}
	
	\subsection{The construction}
	 Choose a decomposition $A_3 = C \oplus L$ where $C\cong R^{q-2}$ and $L \cong R^1$. We next describe the differentials of a complex
	\[
		\mb{B} \colon 0 \to C \xto{\delta_4} A_2 \xto{\delta_3} (A_1^* \otimes L) \oplus A_1 \xto{\delta_3^*} A_2^* \otimes L \xto{\delta_4^*} C^* \otimes L
	\]
	which is self-dual (up to an overall twist).
	\begin{itemize}
		\item The differential $\delta_4$ is the composite $C \xhookrightarrow{i} A_3 \xto{d_3} A_2$, i.e. the $(p+q)\times(q-2)$ submatrix of $d_3$ corresponding to $C \subset A_3$.
		\item Let $\eta\colon A_3 \to L$ be the projection. The differential $\delta_3$ is then given by
		\[
			\delta_3(f) = (\eta((-)\cdot f), \, d_2(f))
		\] 
		for $f \in A_2$. Here the notation $\eta((-)\cdot f)$ means the map sending $e \in A_1$ to $\eta(e\cdot f) \in L$, where we view this map as an element of $A_1^* \otimes L = \Hom(A_1,L)$.
	\end{itemize}
	We write $J = J_C$ for the ideal $I_{q-2}(\delta_4)$. Note that this construction is compatible with localization, or more generally with arbitrary base change preserving acyclicity of $\mb{A}$.
	\begin{remark}
		If $R$ is graded and the differentials of $\mb{A}$ are homogeneous of degree zero, then one may choose a multiplication on $\mb{A}$ that is also homogeneous of degree zero with respect to this internal grading. With that choice, the differentials of $\mb{B}$ will also be homogeneous of degree zero.
	\end{remark}
	
	\begin{lem}\label{lem:B complex}
		As defined above, $\mb{B}$ is a complex.
	\end{lem}
	\begin{proof}
		From the remarks in \S\ref{sec:bg-ffrBE}, it suffices to check this in grade zero, i.e. over localizations $R_\mf{p}$ where $\grade \mf{p} = 0$. In particular, $\grade \mf{p} \leq 2$ so $\mb{A} \otimes R_\mf{p}$ becomes split exact. Thus we may reduce to the case that $\mb{A}$ is split to begin with, and the claim follows from Lemma~\ref{lem:A->B split case}.
	\end{proof}
	The preceding lemma is not difficult to prove directly using the Leibniz rule for the multiplication on $\mb{A}$, which we leave as a fun exercise for the reader. We offer this more ``efficient'' proof because we will leverage Lemma~\ref{lem:A->B split case} to prove the next theorem also.
	
	\iffalse
	\begin{proof}
		Let $g \in C$ and $e \in A_1$. Then $e\cdot g = 0$ since it is an element of $A_4 = 0$, and the Leibniz rule yields
		\[
		0 = d(e\cdot g) = d(e) \cdot g - e \cdot d(g)
		\]
		so in particular $e \cdot d(g) = d(e) \cdot g \in C$. This means that $\eta(e\cdot d(g)) = 0$, hence
		\[
		\delta_3\delta_4(g) = (\eta((-)\cdot d(g)), \, d^2(g))) = 0.
		\]
		Now let $f_1,f_2 \in A_2$. Again, $f_1 \cdot f_2 = 0$ since it is an element of $A_4 = 0$, and the Leibniz rule yields
		\[
		0 = d(f_1 \cdot f_2) = d(f_1) \cdot f_2 + f_1 \cdot d(f_2).
		\]
		So:
		\begin{align*}
			\delta_3^*\delta_3(f_1) &= \delta_2(\eta((-)\cdot f_1), \, d(f_1)) \\
			&= \eta(d(-)\cdot f_1) + \eta(d(f_1)\cdot (-))\\
			&= \eta\big(d(-)\cdot f_1 + d(f_1)\cdot (-)\big) = 0.
		\end{align*}
		Note that the graded-commutativity of the multiplication means that elements in $\mb{A}$ of even homological degree commute with all other elements.
	\end{proof}
	\fi
	
	\begin{thm}\label{thm:A->B}
		Suppose that $\grade J \geq 4$. Then $\mb{B}$ is an acyclic complex. In particular, either $\grade J = 4$ or $J = (1)$.
	\end{thm}
	\begin{proof}
		Just as above, by \S\ref{sec:bg-ffrBE} it suffices to check $\mb{B} \otimes R_\mf{p}$ is acyclic for primes $\mf{p}$ with $\grade \mf{p} \leq 2$, because we already assume that the minors of $\delta_4$ have sufficient grade. For such a prime $\mf{p}$, $\mb{A} \otimes R_\mf{p}$ is split exact. So we again reduce to the case that $\mb{A}$ is split and use Lemma~\ref{lem:A->B split case}.
	\end{proof}
	
	\begin{lem}\label{lem:A->B split case}
		If $\mb{A}$ is split exact, then $\mb{B}$ is also.
	\end{lem}
	\begin{proof}
		Writing $\mb{A}$ in the form \eqref{eq:split A}, by Example~\ref{ex:split A} there exists $b\colon \bigwedge^2 A_1 \to A_3$ so that the multiplication $A_1 \otimes A_2 \to A_3$ is given as follows:
		\begin{align*}
			r_{(1)} \otimes  a_{(2)} &\mapsto a_{(3)}\\
			r_{(1)} \otimes  m_{(2)} &\mapsto b(m_{(1)} \wedge r_{(1)})\\
			m_{(1)} \otimes a_{(2)} &\mapsto 0\\
			m_{(1)} \otimes m'_{(2)} &\mapsto b(m'_{(1)} \wedge m_{(1)})
		\end{align*}
		for $a \in A_3$, $r \in R$, and $m,m' \in M$.
		
		Thus the differential $\delta_3$ has the block form
		\[
		\begin{blockarray}{cccc}
			& C & L & M \\[5pt]
			\begin{block}{c[ccc]}
				R &  &  &  \\
				M &  &  & I \\
				M^* \otimes L &  &  & * \\
				L &  & I & * \\
			\end{block}
		\end{blockarray}
		\]
		where the block for $M \to M^* \otimes L$ is alternating. The differential $\delta_4$ is just the inclusion $C \hookrightarrow C\oplus L \oplus M$, and it is straightforward to see that $\delta_3 \delta_4$ and $\delta_3^* \delta_3$ are both zero. Furthermore, $I_{p+2}(\delta_3) = (1)$ and $\delta_4$ is a split inclusion, so we conclude that $\mb{B}$ is split exact.
	\end{proof}

	\subsection{Existence of a suitable $C \subset F_3$}
	In view of Theorem~\ref{thm:A->B}, it is natural to ask when there exists a choice of $C$ so that $\grade J_C \geq 4$ so that $\mb{B}$ is actually a resolution. This is certainly a restriction, e.g. such a $C$ obviously cannot exist if $\depth R = 3$, but it is only a mild one:
	\begin{prop}
		Let $R$ be a local Noetherian ring with infinite residue field. Suppose that $\mb{A}$ resolves a grade 3 ideal $I \subset R$. Then there exists a choice of $C$ so that $\grade J_C = 4$ if and only if $I$ is generically Gorenstein.
	\end{prop}
	\begin{proof}
		Note that
		\[
			I_{q-1}(d_3) \subseteq J_C \subseteq I_{q-2}(d_3)
		\]
		where the first ideal equals $I$ up to radical and thus has grade 3. The vanishing of $I_{q-2}(d_3)$ cuts out the locus where $\Ext^3(R/I,R)$ cannot be generated by a single element, or equivalently the non-Gorenstein locus for $I$. Clearly $\grade J_C = 4$ implies $\grade I_{q-2}(d_3) \geq 4$ and thus $I$ is generically Gorenstein, establishing the ``only if'' implication.
		
		On the other hand, suppose that every choice of $C$ results in $\grade J_C = 3$. As $I$ is a grade 3 perfect ideal, its minimal primes all have grade 3, and consequently each $J_C$ has a minimal prime in common with $I$. Using that the residue field is infinite, a pigeonhole principle argument produces $C_1,\ldots,C_{q-1}$ which form a basis modulo $\mf{m}$ (as lines in $A_3^*$) and so that $J_{C_i}$ all share the same common minimal prime with $I$. But then $\sum_i J_{C_i} = I_{q-2}(d_3)$ would also have grade 3, showing that $I$ is not generically Gorenstein.
	\end{proof}
	
	\begin{remark}\label{rem:split-in-grade-two-AB}
		Regardless of whether $\grade J = 4$, the output complex $\mb{B}$ will always be split exact in grade 2 by Lemma~\ref{lem:A->B split case}. See also Remark~\ref{rem:split-in-grade-two-BA}.
	\end{remark}
	
	\section{From $\mb{B}$ to $\mb{A}$}\label{sec:grade four to three}
	We now discuss an inverse to the construction given in the preceding section. This is more subtle, and we will need to impose additional hypotheses on the ring $R$. For this section we will assume that $R$ is a unique factorization domain containing $1/2$. It would be interesting to see if a different approach could eliminate the need for this hypothesis.
	
	Throughout, we fix integers $p \geq 1$ and $q\geq 3$. Let $C = R^{q-2}$, $H = R^{p+2}$, and $B_1 = R^{p+q}$.
	First, suppose we have a complex of the form
	\[
		F_1^* \xto{\delta_3} H \oplus H^* \xto{\delta_3^*} F_1.
	\]
	Let $X$ be the part of $\delta_3$ mapping $F_1^* \to H$ and let $Y$ be the part of $\delta_3^*$ mapping $H \to F_1$. Then we have that
	\[
		\delta_3^* \delta_3 = YX + X^*Y^* = 0
	\]
	i.e. $YX$ is skew, though not necessarily alternating. In particular, if $Q$ is the quadratic form on $H \oplus H^*$ given by the evident pairing of the two summands, then for $v \in F_1^*$ we have
	\[
		2Q(v) = 2\langle Xv, Y^*v \rangle_{H} = \langle YXv, v \rangle_{F_1} + \langle v,(YX)^* v \rangle_{F_1} = 0
	\]
	where $\langle -,-\rangle_V$ denotes the evaluation pairing $V \otimes V^* \to R$. Thus if $1/2\in R$ then the image of $\delta_3$ is a (totally) isotropic subspace of $H\oplus H^*$.
	
	Now suppose we have a self-dual resolution
	\[
		0 \to B_0^* \xto{\delta_4} B_1^* \xto{\delta_3} H\oplus H^* \xto{\delta_1} B_1 \xto{\delta_0} B_0
	\]
	and we view the middle module as a (split) quadratic space with the form $Q$. The collection of isotropic summands $R^{p+2} \subset H \oplus H^*$ corresponds to the $R$-points of two copies of the orthogonal Grassmannian:
	\[
		Z = \OG(p+2,2p+4) \amalg \OG'(p+2,2p+4)
	\]
	which we will abbreviate as $\OG$ and $\OG'$. Localizing at $(0)$, the image of $\delta_3$ determines a $\operatorname{Frac}(R)$-point of $Z$; say it lands in $\OG$.
	
	For slightly clearer indexing, let $n = p+2$. We will henceforth assume that the isotropic summand $H^* \subset H\oplus H^*$ represents a point in the \emph{other} component $\OG'$. This is not a serious restriction: if $e_1,\ldots,e_n$ is a basis for $H$ and $e_1',\ldots,e_n' \in H^*$ is its dual basis, we can switch the component containing $H^*$ simply by exchanging $e_n$ with $e_n'$ for example.
	
	\subsection{Spinor structure on $\mb{B}$}
	A spinor structure on $\mb{B}$ will play the role that the DGA structure did in \S\ref{sec:grade three to four}. We refer the reader to \cite{Celikbas-Laxmi-Weyman-23} for details, as we will only give a brief summary here. The group $\Spin(H\oplus H^*)$ is a double cover of $\SO(H\oplus H^*)$, and it has two half-spinor representations
	\[
		V_\mathrm{even} = \bigwedge^\mathrm{even} H, \quad V_\mathrm{odd} = \bigwedge^\mathrm{odd} H.
	\]
	By our assumptions regarding $H^*$ and the image of $\delta_3$, we have Pl\"ucker embeddings $\OG \hookrightarrow \mb{P}(V_\mathrm{odd})$ and $\OG' \hookrightarrow \mb{P}(V_\mathrm{even})$.
	
	\begin{example}
		As discussed previously, the isotropic summand $H^*$ belongs to $\OG'$. Up to scale, its Pl\"ucker coordinates
		\[
			V_\mathrm{even}^* = R \oplus \bigwedge^2 H^* \oplus \cdots \to R
		\]
		are given by 1 on $R$ and 0 on all subsequent components. There is an open patch of $OG'$ given by the graphs of alternating maps $\varphi\colon H^* \to H$; the case of $H^*$ corresponds to $\varphi = 0$. Up to scale, the Pl\"ucker coordinates for such a point in $OG'$ are given by the $n \times n$ Pfaffians of $\varphi$ on the component $\bigwedge^n H^*$.
	\end{example}
	
	Just as how the Pfaffians are square-roots of certain minors, there are spinor coordinates on $\mb{B}$ which are square-roots of certain coordinates of a Buchsbaum-Eisenbud multiplier. Specifically, there is a map $\tilde{\mbf{a}}_3$ so that the diagram
	\[
	\begin{tikzcd}
		R \ar[r, "S_2(\tilde{\mbf{a}}_3)"] \ar[dr,"\mbf{a}_3",swap]& S_2(V_\mathrm{odd}) \ar[d, "\mbf{P}"] \\
		& \bigwedge^{p+2} (H \oplus H^*)
	\end{tikzcd}
	\]
	commutes up to scale by a unit, where $S_2$ denotes the symmetric square, $\mbf{a}_3$ is a map coming from the First Structure Theorem of \cite{Buchsbaum-Eisenbud74}, and $\mbf{P}$ is defined in \cite[Lemma~2.6]{Celikbas-Laxmi-Weyman-23}. We refer to $\tilde{\mbf{a}}_3$ as the spinor coordinates for $\mb{B}$. Note that this map is only up to scale given $\mb{B}$; we a choice.
	
	\begin{remark}
		While \cite[Theorem 4.2]{Celikbas-Laxmi-Weyman-23} assumes that $R$ is regular, this is not necessary, and actually there is a minor issue in the proof since the squares of the spinor coordinates are only equal to coordinates of the Buchsbaum-Eisenbud multiplier up to scale.
		
		Here is a quick argument for why the spinor coordinates exist assuming $R$ is factorial. Let $Z\subset \Spec R$ denote the support of the module resolved by $\mb{B}$, and let $U$ be its complement. Over $U$, the image of $\delta_3$ is an isotropic summand of rank $n$, thus we get a map $U \to \OG\hookrightarrow \mb{P}(V_\mathrm{odd})$. Thus we get spinor coordinates on $U$ which take values in some line bundle, but $\operatorname{Pic}(U) = 0$ because $R$ is factorial. Finally, using the fact that $\grade Z \geq 2$, we know that the restriction map $R \to \Gamma(U, \mc{O}_X)$ is an isomorphism, so the spinor coordinates are defined over $R$. 
	\end{remark}
	
	Having done this, we define one more map $w\colon V_\mathrm{even}^* \to B_1^*$ as follows. Consider the map $V_\mathrm{odd} \otimes (H\oplus H^*) \to V_\mathrm{even}$ given by linearly extending the formula
	\[
		e_1 \wedge \cdots \wedge e_j \otimes (f, f') \mapsto f \wedge e_1 \wedge \cdots \wedge e_j + \sum_i (-1)^{i-1} \langle e_i, f' \rangle e_1 \wedge \cdots \wedge \hat{e}_i \wedge \cdots \wedge e_j
	\]
	for $e_i, f\in H$ and $f'\in H^*$, where $\hat{e}_i$ means to omit $e_i$. Precomposing this map with $\tilde{\mbf{a}}_3 \otimes 1$ yields a map $H \oplus H^* \to V_\mathrm{even}$, whose dual is the diagonal arrow below:
	\begin{equation}\label{eq:w lift}
	\begin{tikzcd}
		0 \ar[r] & B_0^* \ar[r] & B_1^* \ar[r] & H \oplus H^* \ar[r] & B_1\\
		&&V_\mathrm{even}^* \ar[u,dashed] \ar[ur] 
	\end{tikzcd}
	\end{equation}
	We define $w$ to be a lift of this map indicated by the dashed arrow. Any two choices of $w$ differ by $\delta_4 b$ for some $b\colon V_\mathrm{even}^* \to B_0^*$. To prove that such a lift $w$ exists, it is equivalent to check that the image of the diagonal map is a cycle in $H\oplus H^*$. Since this is an equational condition, it can be verified over any dense subset of $\Spec R$. So it is enough to check this in the event that $\mb{B}$ is split exact, which we show now.
	
	\begin{example}\label{ex:split B 1}
		Suppose that $\mb{B}$ is split exact. Let $e_1,\ldots,e_n$ be a basis of $H$ and $e_1',\ldots,e_n' \in H^*$ its dual. Using the action of $\SO(H \oplus H^*)$, we can arrange so that the image of $\delta_3$ is the span of $e_1, e_2',\ldots,e_n'$. Let $m_1,\ldots,m_n$ be a basis of $M$ so that $m_1 \mapsto e_1$ and $m_i \mapsto e_i'$ for $i \geq 2$ below:
		\[
			\mb{B}\colon 0 \to B_0^* \to B_0^* \oplus M \to H \oplus H^* \to \cdots
		\]
		In this situation, the spinor coordinates
		\[
			H^* \oplus \bigwedge^3 H^* \oplus \cdots \xto{\tilde{\mbf{a}}_3^*}  R
		\]
		send $e_1' \in H^*$ to a unit and all other wedge products to zero. We fix the choice that sends $e_1' \mapsto 1$. The diagonal map in \eqref{eq:w lift} then sends
		\begin{align*}
			R \ni 1 &\mapsto e_1 \in H\\
			\bigwedge^2 H^* \ni e_i' \wedge e_1' &\mapsto e_i' \in H^*
		\end{align*}
		and all other wedge products to zero. As such, an evident lift to $B_1^*$ is given by $w$ sending $1 \mapsto m_1$ and $e_i' \wedge e_1' \mapsto m_i$, and all other lifts have the form $w + \delta_4 b$.
	\end{example}
	\begin{example}\label{ex:split B 2}
		As a slight modification to the preceding example, suppose that $\varphi\colon H^* \to H$ is an alternating map, and we instead have $\delta_3(m_1) = e_1$ and $\delta_3(m_i) = e_i' + \varphi(e_i')$ for $i \geq 2$. The higher components of $\tilde{\mbf{a}}_3^*$ and $w$ will change, involving various Pfaffians of $\varphi$. However, their restrictions to the bottom components $H^*$ and $R$ respectively remain the same.
	\end{example}
	
	\begin{lem}\label{lem:singleton-unit}
		Suppose that $R$ is local. Let $H = R^n$, and let $M\cong R^n$ be an isotropic summand of $H \oplus H^*$ that does not belong to the same orbit as $H^*$, so that it has spinor coordinates
		\[
			\lambda \colon \bigwedge^\mathrm{odd} H^* \to R.
		\]
		The following are equivalent:
		\begin{enumerate}
			\item There exists a basis $e_1,\ldots,e_n$ of $H$, dual basis $e_1',\ldots,e_n'$ of $H^*$, and a $(n-1)\times(n-1)$ skew matrix $X$ so that $M$ is spanned by the columns of the matrix
			\[
			\begin{blockarray}{ccc}
				\begin{block}{c[cc]}
					e_1' & 0 & 0\\
					\Span(e_2,\ldots,e_n) & 0 & X\\
					e_1 & 1 & 0\\
					\Span(e_2',\ldots,e_n') & 0 & I\\
				\end{block}
			\end{blockarray}
			\]
			\item $\lambda(H^*) = (1)$.
			\item If $P$ denotes the projection of $M$ onto $H^*$, then $I_{n-1}(P) = (1)$.
		\end{enumerate}
	\end{lem}
	\begin{proof}
		Let us first expand on the condition $\lambda(H^*) = (1)$. It is equivalent to say that there exists a basis $e_1,\ldots,e_n$ of $H$ (and corresponding dual basis of $H^*$) so that $\lambda(e_1')$ is a unit, and $\lambda(e_i') = 0$ for all $i \neq 1$. Now consider $H' = \Span(e_1',e_2,\ldots,e_n)$ and $H'^* = \Span(e_1,e_2',\ldots,e_n')$. We have that
		\[
			\bigwedge^\mathrm{odd} H^* = \bigwedge^\mathrm{even} H'^*
		\]
		where $e_1',\ldots,e_n' \in H^*$ on the left correspond to $1, e_1 \wedge e_2',\ldots,e_1 \wedge e_n'$ on the right. The coordinate $\lambda(1)$ being a unit says that $M$ belongs to the big open cell consisting of isotropic summands that are given by the graph of a skew matrix $H'^* \to H'$. The conditions $\lambda(e_1 \wedge e_i') = 0$ for $i \neq 1$ say that the row and column of this skew matrix corresponding to $e_1$ and $e_1'$ are equal to zero. This establishes the equivalence of (1) and (2).
		
		The implication (1)$\implies$(3) is immediate. For the reverse, there is a basis $e_1,\ldots,e_n$ of $H$ (and corresponding dual basis of $H^*$) so that $M$ is spanned by the columns of the matrix
		\[
		\begin{blockarray}{ccc}
			\begin{block}{c[cc]}
				e_1' & b & 0\\
				\Span(e_2,\ldots,e_n) & -b\mbf{u}^\top & X\\
				e_1 & a & \mbf{u}\\
				\Span(e_2',\ldots,e_n') & 0 & I\\
			\end{block}
		\end{blockarray}
		\]
		where $ab=0$ and $X$ is skew. Our assumption that $M$ belongs to the component of isotropics not containing $H^*$ implies that $b$ cannot be a unit. But the ideal generated by the first column is the unit ideal, so $a$ must be a unit. We may rescale the first column so that $a=1$. Then $b = 0$ and we can set $\mbf{u}=0$ using column operations, establishing (1).
	\end{proof}
	
	\subsection{The construction}\label{subsec:A-from-B}
	Using the original differentials $\delta$ together with the maps $\tilde{\mbf{a}}_3$ and $w$, we define the differentials of a complex
	\[
		\mb{A} \colon 0 \to B_0^* \oplus R \to B_1^* \to H^* \to R.
	\]
	\begin{itemize}
		\item The differential $d_3$ is $\begin{bmatrix}
			\delta_4 & w|_R
		\end{bmatrix}$ viewing $R$ as $\bigwedge^0 H^* \subset V_\mathrm{even}^*$.
		\item The differential $d_2$ is just the dual of $\delta_2$ restricted to $H$.
		\item The differential $d_1$ is $\tilde{\mbf{a}}_3^*$ restricted to $H^* \subset V_\mathrm{odd}^*$.
	\end{itemize}
	Let $I = I_H$ denote the ideal generated by the entries of $d_1$. %We emphasize that it depends on the choice of $H$ and \emph{not} on the choice of $H^*$.
	
	\begin{lem}\label{lem:A complex}
		As constructed above, $\mb{A}$ is a complex.
	\end{lem}
	\begin{proof}
		It suffices to check this at primes $\mf{p}$ of grade zero. Localizing at such a prime, $\mb{B}$ becomes split and the ideal $I_1(\tilde{\mbf{a}}_3)$ generated by \emph{all} the spinor coordinates becomes the unit ideal. Over this localization, acting by a general element of the spin group results in $I = I_1(\tilde{\mbf{a}}_3^*|_{H^*})$ being the unit ideal, so we may reduce further to this setting where the claim is addressed by Lemma~\ref{lem:B->A split case}.
	\end{proof}
	
	\begin{thm}\label{thm:B->A}
		Suppose that $\grade I \geq 3$. Then $\mb{A}$ is an acyclic complex. In particular, either $\grade I = 3$ or $I=(1)$.
	\end{thm}
	\begin{proof}
		Since $\mb{A}$ has length three, it suffices to prove this over localizations $R_\mf{p}$ with $\grade \mf{p} \leq 2$. Localizing at such a prime, we reduce at once to the case that $\mb{B}$ is split exact and $I = (1)$, so we finish using Lemma~\ref{lem:B->A split case}.
	\end{proof}
	
	\begin{lem}\label{lem:B->A split case}
		Suppose that $\mb{B}$ is split exact and $I= (1)$. Then $\mb{A}$ is split exact.
	\end{lem}
	
	\begin{proof}
		By Lemma~\ref{lem:singleton-unit}, there exists a basis $e_1,\ldots,e_n$ of $H$ and an alternating map $\varphi\colon H^* \to H$ so that the image of $\delta_3$ is the span of $e_1$ and $e_i' + \varphi(e_i')$ for $i \geq 2$, where $e_1',\ldots,e_n' \in H^*$ is the dual basis. In this setting, the maps $\tilde{\mbf{a}}_3$ and $w$ are described in Example~\ref{ex:split B 2}, so we may describe the differentials of $\mb{A}$ explicitly, using the same notation as in Example~\ref{ex:split B 1}:
		\[
		d_3 =\begin{blockarray}{cccc}
			& B_0^* & R \\[5pt]
			\begin{block}{c[ccc]}
				B_0^* &I& * \\
				m_1 && 1 \\
				m_2 &&  \\
				\vdots &&  \\
				m_n && \\
			\end{block}
		\end{blockarray},\quad d_2 =\begin{blockarray}{cccccc}
		& B_0^* & m_1 & m_2 & \cdots & m_n \\[5pt]
		\begin{block}{c[ccccc]}
			e_1' &&  \\
			e_2' &&& 1 \\
			\vdots &&&& \ddots \\
			e_n' &&&&& 1\\
		\end{block}
		\end{blockarray}
		\]
		\[
		d_1 =\begin{blockarray}{ccccc}
			& e_1' & e_2' & \cdots & e_n' \\[5pt]
			\begin{block}{c[cccc]}
				R &1& 0 & \cdots &0 \\
			\end{block}
		\end{blockarray}
		\]
		Here the unspecified block in $d_3$ marked with an asterisk comes from the choice of $b$ mentioned at the end of Example~\ref{ex:split B 1}. This complex is evidently split exact, as claimed.
	\end{proof}
	
	\begin{remark}\label{rem:same-radical}
		If $\mb{B}$ is acyclic, then the ideals $I$ and $I_{n-1}(d_2)$ have the same radical. To see this, it suffices to verify they have the same support. It is enough to check this on the complement of the support of $H_0(\mb{B})$, as this set is evidently contained in both. But on the complement, $\mb{B}$ becomes split exact, and the claim then follows from Lemma~\ref{lem:singleton-unit}.
		
		In particular, if $d_2$ is the middle differential of a complex resolving a grade 3 perfect ideal, then necessarily $\grade I_{n-1}(d_2) = 3$. By the preceding, the hypothesis of Theorem~\ref{thm:B->A} holds, and the complex must be isomorphic to $\mb{A}$. This is helpful in practice, as $d_2$ is easier to compute than $d_1$ given $\mb{B}$.
	\end{remark}
	
	\subsection{Existence of a suitable $H$}
	
	\begin{lem}\label{lem:H-prime-avoidance}
		Let $R$ be a local Noetherian ring with infinite residue field, and suppose we have a map
		\[
			\lambda \colon V_\mathrm{odd}^* = H^* \oplus \bigwedge^3 H^* \oplus \cdots \to R
		\]
		whose image is an ideal of grade $\geq 3$. Then, precomposing $\lambda$ by an appropriate element of the spin group acting on $V_\mathrm{odd}^*$, we can arrange for $\grade \lambda(H^*) \geq 3$.
	\end{lem}
	\begin{proof}
		Let $J_k$ denote the image of $\lambda$ restricted to the first $k$ components $H^* \oplus \cdots \oplus \bigwedge^{2k-1} H^*$. Suppose that $\grade J_1 = c < 3$, and let $P_1,\ldots,P_s$ be the grade $c$ associated primes of $J_1$. We will write $\lambda_{1,\ldots,2k-1}$ to mean the coordinate of $\lambda$ on $e'_1 \wedge \cdots \wedge e'_{2k-1}$, and similarly for other indices.
		
		To achieve our goal, we will use the action of certain one-parameter subgroups of the spin group, namely those corresponding to the roots
		\begin{itemize}
			\item $\epsilon_i - \epsilon_j$ ($i \neq j$)
			\item $\epsilon_i + \epsilon_j$ ($i \neq j$).
		\end{itemize}
		The one-parameter subgroups of the first kind generate $\SL(H)$. The action of $\SL(H)$ on the vector $\lambda_1,\ldots,\lambda_n$ is just the standard representation. Since the residue field is infinite, it is well-known that we can arrange (using this action) so that the ideal generated by any set of $c$ coordinates $\lambda_i$ has grade $c$, with grade $c$ associated primes $P_1,\ldots,P_s$.
		
		Let $m$ be the smallest value such that there is a coordinate on $\bigwedge^{2m-1} H^*$ that does not belong to $P_i$ for some $i$. Such an $m$ exists because $\grade J_k \geq 3 > c$ for $k$ sufficiently large. Without loss of generality, say this coordinate is $\lambda_{1,\ldots,2m-1}$.
		
		Acting on $\lambda$ by the one-parameter subgroup associated to $\epsilon_1 + \epsilon_2$ gives $\lambda'$ where $\lambda'_1 = \lambda_1$, $\lambda'_2 = \lambda_2$, and
		\begin{equation}\label{eq:lambdaprime}
			\lambda'_{3,4,\ldots,2m-1} = \lambda_{3,4,\ldots,2m-1} + \alpha \lambda_{1,\ldots,2m-1}
		\end{equation}
		where $\alpha$ denotes the parameter. We will not need to consider what happens to the remaining coordinates.
		
		We set $\alpha = 1$, so that $\lambda'_{3,4,\ldots,2m-1}\notin P_i$. Since $c < 3$, at least $c$ singletons $\lambda_i$ are unchanged, and thus the overall effect is one of the following:
		\begin{itemize}
			\item The set of grade $c$ primes associated to $J_1$ shrinks. If the set disappears entirely, then $c$ increases.
			\item The set of grade $c$ primes associated to $J_1$ stays the same, and $m$ decreases by at least 1 in view of \eqref{eq:lambdaprime}.
		\end{itemize}
		As we repeatedly apply this procedure, neither of these two possibilities can occur indefinitely, so this process must eventually terminate, which can only occur if $c \geq 3$.
	\end{proof}
	\begin{prop}\label{prop:prime-avoidance-H}
		Let $R$ be a local Noetherian ring with infinite residue field, with $1/2 \in R$ and $\Spec R$ irreducible. Let $\mb{B}$ be as in \S\ref{subsec:A-B-notation}, with $B_2$ admitting a hyperbolic basis, and suppose that $\mb{B}$ is acyclic. Then there exists a choice of decomposition $B_2 = H \oplus H^*$ so that the complex $\mb{A}$ constructed in \S\ref{subsec:A-from-B} resolves a grade 3 perfect ideal or the unit ideal.
	\end{prop}
	\begin{proof}
		This follows immediately by applying the preceding lemma in the case that $\lambda = \tilde{\mbf{a}}_3^*$. Note that the ideal $\lambda(H^*)$ only depends on the choice of $H$, not of $H^*$.
	\end{proof}
	It is interesting to note that, in the proof of Lemma~\ref{lem:H-prime-avoidance}, we rely on the two coordinates $\lambda_{i}, \lambda_{j}$ being unchanged when acting by the one-parameter subgroup associated to $\epsilon_i + \epsilon_j$. So the argument does not work for increasing $\grade \lambda(H^*)$ to 4, even if it is known that the image of $\lambda$ has grade 4.
	
	The case of $\lambda = \tilde{\mbf{a}}_3^*$ shows that this is not just a limitation of the argument, but rather it may be impossible: $\grade I \leq \pdim R/I = 3$ in that case.
	
	\begin{remark}\label{rem:split-in-grade-two-BA}
		The construction of $\mb{A}$ from $\mb{B}$ works even if we only assume that $\mb{B}$ is split exact in grade 2, rather than fully acyclic. With this weaker hypothesis, the row in \eqref{eq:w lift} is still exact, and the spinor coordinates would still generate an ideal of grade $\geq 3$, so the conclusion of Proposition~\ref{prop:prime-avoidance-H} holds. See also Remark~\ref{rem:split-in-grade-two-AB}.
	\end{remark}
	
	\section{Examples}\label{sec:examples}
	In the preceding sections, we have demonstrated the constructions in the simplest possible case, namely when the input is a split exact complex. This was sufficient to prove that the constructions are valid in general. In a future paper, we intend to thoroughly investigate the case when the input is a resolution of an ideal in the linkage class of a complete intersection (licci), with the aim of better understanding the relationship between grade 3 licci ideals of type 2 and grade 4 licci Gorenstein ideals.
	
	At present, we conclude with a partial demonstration of \S\ref{sec:grade four to three} on the Gulliksen-Negard resolution of $(n-1) \times (n-1)$ minors of a generic $n\times n$ matrix. We thank Oana Veliche for helping find a suitable isotropic subspace $H$ for $n=3$.
	\begin{example}
		Let $X = (x_{ij})$ be a generic $n \times n$ matrix, and let $R = \mb{C}[\{x_{ij}\}]$. Following \cite{Gulliksen-Negard72}, the resolution of $R/I_{n-1}(X)$ can be described in the following manner. We view the free modules in the resolution as follows:
		\begin{itemize}
			\item $B_1, B_3$ consists of $n\times n$ matrices.
			\item $B_2$ consists of pairs $(A,B)$ of $n \times n$ matrices where $\operatorname{tr}(A) = \operatorname{tr}(B)$, modulo $\Span(I,I)$.
		\end{itemize}
		The differentials can be described as:
		\begin{itemize}
			\item $\delta_1(A) = \operatorname{tr}(A \operatorname{adj}(X))$
			\item $\delta_2(A,B) = AX-XB$
			\item $\delta_3(A) = (XA,AX)$
			\item $\delta_4(1) = \operatorname{adj}(X)$
		\end{itemize}
		and the symmetric bilinear form
		\begin{equation}\label{eq:GN-bilinear-form}
			\langle (A,B),(C,D) \rangle = \operatorname{tr}(AC-BD)
		\end{equation}
		on $B_2$ induces a isomorphism $B_2^* \cong B_2$ that can be used to make the complex explicitly self-dual.
		
		Let $e_{ij}$ denote the matrix with a single 1 in position $(i,j)$, and all other entries 0. Consider the following basis of $B_2$:
		\[
		\begin{array}{ccc}
			\left\{(e_{ij},0)\right\}_{1\leq i < j \leq n} & \left\{(0,e_{ji})\right\}_{1\leq i < j \leq n} & \left\{(e_{ii},e_{ii})\right\}_{1\leq i \leq n-1}\\
			\left\{(e_{ji},0)\right\}_{1\leq i < j \leq n} & \left\{(0,-e_{ij})\right\}_{1\leq i < j \leq n} & \left\{\frac{1}{2} (e_{ii}-e_{nn}, e_{nn}-e_{ii})\right\}_{1\leq i \leq n-1}
		\end{array}
		\]
		This basis is homogeneous with respect to the evident $\mb{Z}^{2n}$ multigrading on $R$. Using \eqref{eq:GN-bilinear-form}, it is straightforward to check that this is a hyperbolic basis: if we let $H_0$ be the isotropic subspace with basis given by the upper row, then we may take $H_0^*$ to be the span of the lower row, with the lower row giving the dual basis.
		
		Consider the skew map $H_0 \to H_0^*$ sending $(e_{ij},0) \mapsto -(0,-e_{ij})$, $(0,e_{ji}) \mapsto (e_{ji},0)$, and other basis elements to 0. Then the graph of this skew map is $\{(A,A)\} \subset B_2$. Notice that if consider the homomorphism $R \to \mb{C}$ specializing $X\mapsto I$, then this would be none other than the image of $\delta_3$ after specialization. Therefore $H_0$ and $\operatorname{im}(\delta_3)$ belong to the same component of isotropics.
		
		Suppose that $n=3$. Let
		\begin{align*}
			H &= \Span\Big\{(e_{21},0),(e_{32},0),(0,-e_{12}),(0,-e_{23}),\\
			&\quad (e_{13},0),(0,e_{31}),(e_{22},e_{22}),\frac{1}{2}(e_{11}-e_{33},e_{33}-e_{11})\Big\}.
		\end{align*}
		Comparing to the bases of $H_0$ and $H_0^*$ written before, one can verify that if $H^*$ is any complementary isotropic to $H$, then $H^*$ and $H_0$ are in different components of isotropics.
		
		Furthermore, the hypothesis of Theorem~\ref{thm:B->A} is met with this choice of $H$. Note that we do not need to compute the ideal $I$ of spinor coordinates directly, as by Remark~\ref{rem:same-radical} it is enough to verify that $d_2$ is plausibly the middle differential of a length three resolution, and this can be checked in Macaulay2.
		
		Explicitly, restricting $\delta_2$ to $H$ and taking the dual gives us
		\[
		d_2 = \left(\!\begin{array}{ccccccccc}
			0&x_{1,1}&0&0&x_{1,2}&0&0&x_{1,3}&0\\
			0&0&x_{2,1}&0&0&x_{2,2}&0&0&x_{2,3}\\
			0&0&0&x_{1,1}&x_{2,1}&x_{3,1}&0&0&0\\
			0&0&0&0&0&0&x_{1,2}&x_{2,2}&x_{3,2}\\
			x_{3,1}&0&0&x_{3,2}&0&0&x_{3,3}&0&0\\
			-x_{1,3}&-x_{2,3}&-x_{3,3}&0&0&0&0&0&0\\
			0&x_{2,1}&0&-x_{1,2}&0&-x_{3,2}&0&x_{2,3}&0\\
			x_{1,1}&x_{2,1}/2&0&x_{1,2}/2&0&-x_{3,2}/2&0&-x_{2,3}/2&-x_{3,3}
		\end{array}\!\right)
		\]
		and by computing the kernel of $d_2$ and $d_2^*$, we get
		\begin{gather*}
		d_3 = \left(\!\begin{array}{cc}
			-x_{2,3}x_{3,2}+x_{2,2}x_{3,3}&0\\
			x_{1,3}x_{3,2}&x_{1,2}x_{3,3}\\
			-x_{1,3}x_{2,2}&-x_{1,2}x_{2,3}\\
			x_{2,3}x_{3,1}&x_{2,1}x_{3,3}\\
			-x_{1,3}x_{3,1}&-x_{1,1}x_{3,3}\\
			x_{1,3}x_{2,1}-x_{1,1}x_{2,3}&0\\
			-x_{2,2}x_{3,1}&-x_{2,1}x_{3,2}\\
			x_{1,2}x_{3,1}-x_{1,1}x_{3,2}&0\\
			x_{1,1}x_{2,2}&x_{1,2}x_{2,1}
		\end{array}\!\right)\\
		d_1^* = \left(\!\begin{array}{c}
		-x_{2,1}x_{2,3}x_{3,2}+x_{2,1}x_{2,2}x_{3,3}\\
		x_{1,2}x_{3,1}x_{3,3}-x_{1,1}x_{3,2}x_{3,3}\\
		x_{1,2}x_{2,3}x_{3,2}-x_{1,2}x_{2,2}x_{3,3}\\
		-x_{1,3}x_{2,1}x_{3,3}+x_{1,1}x_{2,3}x_{3,3}\\
		x_{1,2}x_{1,3}x_{2,1}-x_{1,1}x_{1,2}x_{2,3}\\
		x_{1,2}x_{2,1}x_{3,1}-x_{1,1}x_{2,1}x_{3,2}\\
		x_{1,2}x_{2,3}x_{3,1}/2+x_{1,3}x_{2,1}x_{3,2}/2-x_{1,1}x_{2,2}x_{3,3}\\
		x_{1,2}x_{2,3}x_{3,1}-x_{1,3}x_{2,1}x_{3,2}
		\end{array}\!\right)
		\end{gather*}
		and Macaulay2 verifies that this assembles into a length three resolution.
		
		It would be interesting to see whether there is a systematic choice of $H$ that works for arbitrary $n$. Ideally this choice of $H$ should either be easy to work with algebraically, or have some sort of geometric significance.
	\end{example}

	\printbibliography
\end{document}